\documentclass[11pt]{amsart}
\usepackage{amssymb}
\usepackage{verbatim}
\usepackage{hyperref}

\begin{document}

\newtheorem{theorem}{Theorem}    
\newtheorem{proposition}[theorem]{Proposition}
\newtheorem{conjecture}[theorem]{Conjecture}
\def\theconjecture{\unskip}
\newtheorem{corollary}[theorem]{Corollary}
\newtheorem{lemma}[theorem]{Lemma}
\newtheorem{sublemma}[theorem]{Sublemma}
\newtheorem{fact}[theorem]{Fact}
\newtheorem{observation}[theorem]{Observation}
\theoremstyle{definition}
\newtheorem{definition}{Definition}
\newtheorem{notation}[definition]{Notation}
\newtheorem{remark}[definition]{Remark}
\newtheorem{question}[definition]{Question}
\newtheorem{questions}[definition]{Questions}

\newtheorem{example}[definition]{Example}
\newtheorem{problem}[definition]{Problem}
\newtheorem{exercise}[definition]{Exercise}


\def\reals{{\mathbb R}}
\def\torus{{\mathbb T}}
\def\heis{{\mathbb H}}
\def\integers{{\mathbb Z}}
\def\rationals{{\mathbb Q}}
\def\naturals{{\mathbb N}}
\def\complex{{\mathbb C}\/}
\def\distance{\operatorname{distance}\,}
\def\sym{\operatorname{Symm}\,}
\def\support{\operatorname{support}\,}
\def\dist{\operatorname{dist}\,}
\def\Span{\operatorname{span}\,}
\def\degree{\operatorname{degree}\,}
\def\kernel{\operatorname{kernel}\,}
\def\dim{\operatorname{dim}\,}
\def\codim{\operatorname{codim}}
\def\trace{\operatorname{trace\,}}
\def\Span{\operatorname{span}\,}
\def\dimension{\operatorname{dimension}\,}
\def\codimension{\operatorname{codimension}\,}
\def\Gl{\operatorname{Gl}\,}
\def\nullspace{\scriptk}
\def\kernel{\operatorname{Ker}}
\def\ZZ{ {\mathbb Z} }
\def\p{\partial}
\def\rp{{ ^{-1} }}
\def\Re{\operatorname{Re\,} }
\def\Im{\operatorname{Im\,} }
\def\ov{\overline}
\def\eps{\varepsilon}
\def\lt{L^2}
\def\diver{\operatorname{div}}
\def\curl{\operatorname{curl}}
\def\etta{\eta}
\newcommand{\norm}[1]{ \|  #1 \|}
\def\expect{\mathbb E}
\def\bull{$\bullet$\ }
 \def\newbull{\newline$\bullet$\ }
 \def\nobull{\noindent$\bullet$\ }
\def\det{\operatorname{det}}
\def\Det{\operatorname{Det}}
\def\multiR{\mathbf R}
\def\bestA{\mathbf A}
\def\bestB{\mathbf B}
\def\Apq{\mathbf A_{p,q}}
\def\Apqr{\mathbf A_{p,q,r}}
\def\rank{\operatorname{rank}}
\def\rankk{\mathbf r}
\def\diameter{\operatorname{diameter}}
\def\bp{\mathbf p}
\def\bff{\mathbf f}
\def\bg{\mathbf g}
\def\essd{\operatorname{essential\ diameter}}

\def\mab{M}
\def\t2{\tfrac12}

\newcommand{\abr}[1]{ \langle  #1 \rangle}
\def\unitQ{{\mathbf Q}}
\def\mbfp{{\mathbf P}}

\def\aff{\operatorname{Aff}}
\def\wb{\mathbf C}

\newcommand{\Norm}[1]{ \Big\|  #1 \Big\| }
\newcommand{\set}[1]{ \left\{ #1 \right\} }
\def\one{{\mathbf 1}}
\newcommand{\modulo}[2]{[#1]_{#2}}

\def\rint{ \int_{\reals^+} }
\def\Abest{{\mathbb A}}

\def\barrier{\medskip\hrule\hrule\medskip}

\def\scriptf{{\mathcal F}}
\def\scripts{{\mathcal S}}
\def\scriptq{{\mathcal Q}}
\def\scriptg{{\mathcal G}}
\def\scriptm{{\mathcal M}}
\def\scriptb{{\mathcal B}}
\def\scriptc{{\mathcal C}}
\def\scriptt{{\mathcal T}}
\def\scripti{{\mathcal I}}
\def\scripte{{\mathcal E}}
\def\scriptv{{\mathcal V}}
\def\scriptw{{\mathcal W}}
\def\scriptu{{\mathcal U}}
\def\scripta{{\mathcal A}}
\def\scriptr{{\mathcal R}}
\def\scripto{{\mathcal O}}
\def\scripth{{\mathcal H}}
\def\scriptd{{\mathcal D}}
\def\scriptl{{\mathcal L}}
\def\scriptn{{\mathcal N}}
\def\scriptp{{\mathcal P}}
\def\scriptk{{\mathcal K}}
\def\scriptP{{\mathcal P}}
\def\scriptj{{\mathcal J}}
\def\scriptz{{\mathcal Z}}
\def\frakv{{\mathfrak V}}
\def\frakG{{\mathfrak G}}
\def\frakA{{\mathfrak A}}
\def\frakB{{\mathfrak B}}
\def\frakC{{\mathfrak C}}
\def\frakf{{\mathfrak F}}
\def\fcross{{\mathfrak F^{\times}}}

 \def\ccount{{\mathbf r}}

\author{Michael Christ}
\address{
        Michael Christ\\
        Department of Mathematics\\
        University of California \\
        Berkeley, CA 94720-3840, USA}
\email{mchrist@berkeley.edu}

\date{August 25, 2015.}
\thanks{Research supported by NSF grant DMS-1363324.}

\title
[On an apparently bilinear inequality  for the Fourier transform]
{On an apparently bilinear inequality \\ for the Fourier transform}

\begin{abstract}
A bilinear inequality of Geba et.\ al.\ for the Fourier transform 
is shown to be equivalent to a simpler linear inequality, and the range
of exponents is extended.
Related mixed-norm inequalities are discussed.
\end{abstract}
\maketitle

This note is a comment on a bilinear inequality of Geba et.\ al.\ \cite{greenleaf}.
Our thesis is that a simpler and purely linear inequality 
underlies the bilinear inequality.
This reasoning extends the bilinear inequality to the full range of
allowable exponents.
The linear inequality itself is not new; indeed, it is proved
in \cite{greenleaf} as a corollary of the bilinear one.

\section{Linear versus bilinear}

View $\reals^d$ as $\reals^{d'}\times\reals^{d''}$
with coordinates $x=(x',x'')$ and $\xi = (\xi',\xi'')$.
Consider the Fourier transform
\[\widehat{f}(\xi) = \int e^{-2\pi ix\cdot\xi} f(x)\,dx
= \int_{\reals^{d'}}
e^{-2\pi i x'\cdot\xi'} 
\int_{\reals^{d''}}
e^{-2\pi i x''\cdot\xi''} 
F(x',x'')\,dx''\,dx'.\]
Consider mixed-norm spaces $L^pL^s = L^p_{x'}L^s_{x''}$ with norms
$\norm{F}_{p,s}= \norm{F}_{L^p_{x'}L^s_{x''}}$ defined by first taking the $L^s$ norm with respect to
$x''$ for each fixed $x'$, then the $L^p$ norm of the resulting function of $x'$.
The spaces 
$L^s_{x''}L^p_{x'}$, in which the order in which the individual norms is taken is reversed,
also play a part in our discussion.
All inequalities in this note can be regarded as {\it a priori} inequalities for Schwartz functions.

Define \[\wb_r = r^{1/2r} t^{-1/2t}\] where $t=r'= r/(r-1)$ is the exponent conjugate to $r$.
Then $\wb_r<1$ for all $r\in(1,2)$.
Beckner \cite{beckner} has shown that $\wb_p^n$ is the optimal constant
in the Hausdorff-Young inequality. That is, $\norm{\widehat{f}}_{L^{p'}(\reals^n)}
\le \wb_p^n \norm{f}_{L^p(\reals^n)}$ for all $p\in[1,2]$
and all $n\ge 1$, and the inequality holds with no smaller 
constant factor on the right-hand side.

The linear inequality in question is as follows.
\begin{proposition} \label{prop1}
For all $p\in[1,2]$ and all $F$,
\begin{equation} \label{restriction}
\norm{\widehat{F}\big|_{\xi''=0}}_{L^{p'}(\reals^{d'})} \le
\wb_p^{d'} \norm{F}_{p,1}.\end{equation}
\end{proposition}

\begin{proof}
Define \[ f(x') = \int  F(x',x'')\,dx''.\]
If $F\in L^pL^1$ then $f\in L^p(\reals^{d'})$
and $\norm{f}_{L^p}\le \norm{F}_{p,1}$. Moreover
\begin{align*}
\widehat{F}(\xi',0) = \iint e^{-2\pi ix'\cdot\xi'} F(x',x'')\,dx''\,dx'
= \int e^{-2\pi ix'\cdot\xi'} f(x')\,dx'
= \widehat{f}(\xi').
\end{align*}
It now suffices to invoke the Hausdorff-Young inequality.
\end{proof}

The bilinear inequality of \cite{greenleaf} states that for all Schwartz functions,
\begin{equation} \label{bilinear} \norm{\widehat{FG}\big|_{\xi''=0}}_{L^r}
\le C\norm{F}_{p,s}\norm{G}_{q,t}\end{equation}
for all $(p,s;q,t;r)$
satisfying 
\begin{equation} \label{necessary}
s^{-1}+t^{-1}=1 \ \text{ and } \  r^{-1} = 1-p^{-1}-q^{-1} \end{equation}
with the supplementary restrictions $p,q,r\ge 2$ and $s=t=2$.
The $L^r$ norm is that of $\reals^{d'}$.
The relations \eqref{necessary} are necessary for the inequality to hold, 
as observed in \cite{greenleaf}.\footnote{
There are typographical inaccuracies in the formulas relating exponents in inequalities 
(1.10) and (1.11) of \cite{greenleaf}.} 
It is likewise necessary that $r\ge 2$.
 
The next result states that inequality \eqref{bilinear}, 
for the full range of possible exponents satisfying \eqref{necessary}
with no supplementary restrictions,
follows from the linear inequality \eqref{restriction}. 
Conversely, \eqref{restriction} follows from
\eqref{bilinear} for the more limited range of exponents treated in \cite{greenleaf},
as already noted in \cite{greenleaf}.

\begin{corollary}
If $p,q\in[1,\infty]$, $r\in[2,\infty]$, $s,t\in[1,\infty]$,
and $(p,s,q,t,r)$ satisfies \eqref{necessary} 
then the inequality \eqref{bilinear} holds
with $C = \wb_{r'}^{d'}$.
\end{corollary}

\begin{proof}
If $F\in L^pL^s$ and $G\in L^q L^t$ then by H\"older's inequality,
$FG\in L^uL^v$ where $u^{-1} = p^{-1}+q^{-1}$ and $v^{-1} = s^{-1}+t^{-1}$.
The assumption $r\ge 2$ is equivalent to $u\in[1,2]$, while $s^{-1}+t^{-1}=1$ means that $v=1$.
\end{proof}



\section{Mixed norm inequalities}

A mixed norm inequality $\norm{\widehat{F}}_{p',\infty} \le C\norm{F}_{p,1}$
would not only imply \eqref{restriction}, but would be dramatically stronger.
Therefore it is natural to ask whether such inequalities are valid.
We next discuss inequalities with mixed norms for $\widehat{F}$ on the left-hand side.

Given a function of $(y',y'')$, one can form mixed norms
$\norm{F}_{L^q_{y'}L^t_{y''}}$
in which an $L^t$ norm is taken of $y''\mapsto F(y',y'')$
for fixed $y'$, and then the $L^q$ norm of the resulting
function of $y'$ is taken. 
One can equally well reverse the order,
$\norm{F}_{L^6_{y''}L^q_{y'}}$
taking the $L^q$ norm first with respect to $y'$ and then
the $L^t$ norm with respect to $y''$.
These two quantities are related by
Minkowski's integral inequality: 
\begin{equation} \label{Mink}
\norm{F}_{L^q_{y'}L^t_{y''}}\le \norm{F}_{L^t_{y''}L^q_{y'}}
\ \text{ whenever } t\ge q.  \end{equation}

\begin{proposition} \label{prop:variant}
For all $p,s\in[1,2]$,
\begin{equation} \label{optimal}
\norm{\widehat{F}}_{L^{s'}_{\xi''}L^{p'}_{\xi'}}  
\le \wb_p^{d'}\wb_s^{d''}\norm{F}_{L^p_{x'} L^s_{x''}}.\end{equation}
\end{proposition}

Note that the order of the norms on the left-hand side is reversed;
the $L^{p'}_{\xi'}$ norm is taken first.

\begin{proof}
Regard the Fourier transform for $\reals^d$
as the composition of a partial Fourier transform in the second variable,
that is, in $\reals^{d''}$, followed by a partial Fourier transform in the first
variable. The first operation maps $L^p_{x'}L^s_{x''}$ to 
$L^p_{x'}L^{s'}_{\xi''}$
with operator norm $\wb_s^{d''}$. 
Since $p\le 2\le s'$, \eqref{Mink} says that the space
$L^p_{x'}L^{s'}_{\xi''}$
is contained in $L^{s'}_{\xi''}L^p_{x'}$ with a contractive inclusion map. 
The second operation maps this last mixed-norm space to
$L^{s'}_{\xi''}L^{p'}_{\xi'}$, with operator norm $\wb_p^{d'}$.
\end{proof}

It is likewise natural to seek inequalities with
$\norm{\widehat{F}}_{L^{p'}_{\xi'}L^{s'}_{\xi''}}$ on the left-hand side; 
now there is no reversal of the ordering.
Such an inequality with $s=1$ would be far stronger than \eqref{restriction}.
For certain pairs of exponents, such inequalities do hold,  as direct
consequences of Proposition~\ref{prop:variant}, 

\begin{corollary} \label{cor2}
If $p,s\in[1,2]$ and $p\le s$ then 
\begin{equation} \label{optimal2}
\norm{\widehat{F}}_{p',s'} \le \wb_p^{d'}\wb_s^{d''}\norm{F}_{p,s}. \end{equation}
\end{corollary}

\begin{proof}
$p'\ge s'$,  so by \eqref{Mink},
$\norm{\widehat{F}}_{L^{p'}_{\xi'}L^{s'}_{\xi''}}  
\le \norm{\widehat{F}}_{L^{s'}_{\xi''}L^{p'}_{\xi'}}$.
Invoke Proposition~\ref{prop:variant}.
\end{proof}

However, the extension of Corollary~\ref{cor2} 
to $s<p$ breaks down, no matter how large a constant
factor is allowed on the right-hand side.

\begin{proposition}
Let $d',d''$ both be strictly positive integers, and let $s<p$. Then
the Fourier transform fails to map $L^p_{x'}L^s_{x''}(\reals^{d'}\times\reals^{d''})$
boundedly to $L^{p'}_{\xi'} L^{s'}_{\xi''}(\reals^{d'}\times\reals^{d''})$. 
\end{proposition}

\begin{proof}[Sketch of proof]
First consider $s=1$.
Then $s'=\infty$ and if such an inequality were valid then by a limiting argument 
one could choose $F(x',x'') = f(x')\delta_{x''=\psi(x')}$, for an arbitrary measurable
function $\psi:\reals^{d'}\to\reals^{d''}$, to conclude that
\[ \norm{\int_{\reals^{d'}} e^{-2\pi i\phi(\xi')\cdot\psi(x')} e^{-2\pi i x'\cdot\xi'} f(x')\,dx'}_{p'} 
\le C \norm{f}_p,\] 
uniformly for all $f\in L^p(\reals^{d'})$
and all Lebesgue measurable functions $\phi,\psi:\reals^{d'}\to\reals^{d''}$. 
This is absurd, as choosing $\phi(\xi')\equiv \xi'$ and $\psi(x')\equiv -x'$ makes plain.

Next consider the case $1<s<p\le 2$.
To simplify notation consider the case $d'=d''=1$.
Let $f,g:\reals^1\to\complex$ be Schwartz functions that do not vanish identically.
For $t\in\reals^+$ consider $F(x,y) = f_t(x)g(y-x)$
where $f_t(x) = t^{1/p} f(tx)$. 
Then
\begin{align*} \widehat{F}(\xi,\eta)
&= \int_{\reals}
e^{-2\pi i x\xi} f_t(x)
\int_{\reals}
e^{-2\pi i y\eta} g(y-x)
\,dy\,dx
\\&
= \widehat{g}(\eta) \int_{\reals}
e^{-2\pi i x(\xi-\eta)}  t^{1/p} f(tx)
\,dx
\\&
= \widehat{g}(\eta) t^{-1/p'} \widehat{f}(t^{-1}(\xi-\eta)).
\end{align*}
It is straightforward to verify that as $t\to 0$, the $L^{p'}_{\xi}L^{s'}_\eta$ norm of 
this function has order of magnitude $t^{-1/p'} t^{1/s'}$.
If $s<p$, $1/s' < 1/p'$ and consequently 
$t^{-1/p'} t^{1/s'}\to\infty$ as $t\to 0$.

The case of arbitrary dimensions can be reduced to this construction for $\reals^1\times\reals^1$ 
by consideration of functions that are
products of the individual coordinates, with $f_t$ defined by dilation with respect to a single
coordinate and $g(y-x)$ replaced by $g(x''-vx'_i)$ for some nonzero vector $v\in\reals^{d''}$
and some coordinate $x'_i$ of $x'$.
\end{proof}

\medskip
The constant factors on the right-hand sides
of the inequalities in the first four results stated above are all optimal.
To demonstrate this, it suffices to consider
product functions $F(x',x'') = g(x')h(x'')$ and to invoke the optimality
of the constants in Beckner's inequality.

\medskip
The formulations and reasoning extend to a product of any two locally compact Abelian groups,
with norms defined with respect to Haar measure and Fourier transforms normalized
so as to be unitary from $\lt$ to $\lt$ of the dual groups,
and contractive from $L^1$ to $L^\infty$ of the dual groups.
The constants $\wb_p^{d'}$ and $\wb_p^{d'}\wb_s^{d''}$ 
in the above results are replaced by 
the optimal constant in the Hausdorff-Young inequality for the first factor,
and by the product of the optimal constants for the two factor groups, respectively.


\begin{thebibliography}{20}

\bibitem{beckner} 
W.~Beckner, {\em Inequalities in Fourier analysis}, Ann. of Math. (2) 102 (1975), no. 1, 159--182 

\bibitem{greenleaf}
D.~Geba, A.~Greenleaf, A.~Iosevich, E.~Palsson and E.~Sawyer,
{\em Restricted convolution inequalities, multilinear operators and applications},
Math. Res. Lett. 20 (2013), no. 4, 675-694.


\end{thebibliography}
\end{document}